\documentclass[12pt]{amsart}
\usepackage{amsmath,amssymb,bbm,esint,times,url}

\usepackage{tikz}
\usetikzlibrary{calc}
\usepackage[bookmarks=false,unicode,colorlinks,urlcolor=red,citecolor=red,linkcolor=blue]{hyperref}
\usepackage{aliascnt}
\usepackage{booktabs}
\usepackage[margin=1.25in]{geometry}

\newtheorem{theorem}{Theorem}[section]

\newaliascnt{lemma}{theorem}
\newtheorem{lemma}[lemma]{Lemma}
\aliascntresetthe{lemma}

\newaliascnt{remark}{theorem}
\newtheorem{remark}[remark]{Remark}
\aliascntresetthe{remark}

\newaliascnt{proposition}{theorem}

\aliascntresetthe{proposition}

\newaliascnt{corollary}{theorem}
\newtheorem{corollary}[corollary]{Corollary}
\aliascntresetthe{corollary}

\newaliascnt{conjecture}{theorem}
\newtheorem{conjecture}[conjecture]{Conjecture}
\aliascntresetthe{conjecture}

\makeatletter
\def\tagform@#1{\maketag@@@{\ignorespaces#1\unskip\@@italiccorr}}
\let\orgtheequation\theequation
\def\theequation{(\orgtheequation)}
\makeatother
\def\equationautorefname~{}

\newcommand{\arxiv}[1]{%
 \href{http://front.math.ucdavis.edu/#1}{ArXiv:#1}}
\newcommand{\mref}[1]{%
\href{http://www.ams.org/mathscinet-getitem?mr=#1}{#1}}

\begin{document}

\title[]{On mixed Dirichlet-Neumann eigenvalues of triangles.}
\author[]{Bart{\l}omiej Siudeja}
 
\begin{abstract}
We order lowest mixed Dirichlet-Neumann eigenvalues of right triangles according to which sides we apply the Dirichlet conditions. It is generally true that Dirichlet condition on a superset leads to larger eigenvalues, but it is nontrivial to compare e.g. the mixed cases on triangles with just one Dirichlet side. As a consequence of that order we also classify the lowest Neumann and Dirichlet eigenvalues of rhombi according to their symmetry/antisymmetry with respect to the diagonal.

We also give an order for the mixed Dirichlet-Neumann eigenvalues on arbitrary triangle, assuming two Dirichlet sides. The single Dirichlet side case is conjectured to also have appropriate order, following right triangular case.
\end{abstract}

\maketitle

\section{Introduction}
Laplace eigenvalues are often interpreted as frequencies of vibrating membranes. In this context, the natural (Neumann) boundary condition corresponds to a free membrane, while Dirichlet condition indicates a membrane is fixed in place on the boundary. Intuitively, mixed Dirichlet-Neumann conditions should mean that the membrane is partially attached, and the larger the attached portion, the higher the frequencies.

Using variational characterization of the frequencies (see \autoref{sec:def}) one can easily conclude that increasing the attached portion leads to increased frequencies. In this paper we investigate a harder, yet still intuitively clear case of imposing Dirichlet conditions on various sides of triangles. Imposing Dirichlet condition on one side gives smaller eigenvalues than imposing it on that side and one more. However, is it true that imposing Dirichlet condition on shorter side leads to smaller eigenvalue than the Dirichlet condition on a longer side? 

Note that one can also think about eigenvalues as related to the survival probability of the Brownian motion on a triangle, reflecting on the Neumann boundary, and dying on the Dirichlet part. In this context, it is clear that enlarging the Dirichlet part leads to shorter survival time. It is also reasonable, that having Dirichlet condition on one long side gives larger chance of dying, than having shorter Dirichlet side. However, this case is far from obvious to prove, especially that the difference might be very small for nearly equilateral triangles.

Let $L$, $M$ and $S$ denote the lengths of the sides of a triangle $T$, so that $L\ge M\ge S$. Let the smallest eigenvalue corresponding to Dirichlet condition applied to a chosen set of sides be denoted by $\lambda_1^{set}$. E.g. $\lambda_1^{LS}$ would correspond to Dirichlet condition imposed on the longest and shortest sides. Let also $\mu_2$ and $\lambda_1$ denote the smallest nonzero pure Neumann and pure Dirichlet eigenvalues of the same triangles.

\begin{theorem} \label{thm:order}
  For any right triangle with smallest angle satisfying $\pi/6<\alpha<\pi/4$ 
  \begin{align*}
    0=\mu_1&<\lambda_1^S<\lambda_1^M<\mu_2< \lambda_1^L< \lambda_1^{MS}<\lambda_1^{LS}<\lambda_1^{LM}<\lambda_1.
  \end{align*}
  When $\alpha=\pi/6$ (half-of-equilateral triangle) we have $\lambda_1^M=\mu_2$, and for $\alpha=\pi/4$ (right isosceles triangle) we have $S=M$ and $\lambda_1^L=\mu_2$. All other inequalities stay sharp in these cases.

  Furthermore for arbitrary triangle
  \begin{align*}
    \min\{\lambda_1^S,\lambda_1^M,\lambda_1^L\}<\mu_2\le \lambda_{1}^{MS}<\lambda_1^{LS}<\lambda_1^{LM},
  \end{align*}
  as long as the appropriate sides have different lengths. However, it is possible that $\mu_2>\lambda_1^L$ (for any small perturbation of the equilateral triangle) or $\mu_2<\lambda_1^M$ (for right triangles with $\alpha<\pi/6$).
\end{theorem}

Note that, for arbitrary polygonal domains, it is not always the case that a longer restriction leads to a higher eigenvalue (see \autoref{rem:trap}). The theorem also asserts that a precise position of the smallest nonzero Neumann eigenvalue in the ordered sequence is an exception, rather than a rule (even among triangles). Nevertheless, we conjecture that mixed eigenvalues of triangles can be fully ordered. More precisely, we conjecture that all cases missing in the above theorem are still true:
\begin{conjecture}
  For arbitrary triangle
  \begin{align*}
    \lambda_1^S<\lambda_1^M<\lambda_1^L<\lambda_1^{MS},
  \end{align*}
  as long as appropriate sides have different lengths.
\end{conjecture}

Even though right triangles are a rather special case, they are of interest in studying other polygonal domains. In particular, recent paper by Nitsch \cite{Ni14} studies regular polygons via eigenvalue perturbations on right triangles. Similar approach is taken in the author's upcoming collaboration \cite{NSY}. Finally, right triangles play the main role in the recent progress on the celebrated hot-spots conjecture. Newly discovered approach due to Miyamoto \cite{Mi09,Mi13} led to new partial results for acute triangles \cite{Shot} (see also Polymath 7 project \href{http://polymathprojects.org/tag/polymath7/}{polymathprojects.org/tag/polymath7/}). The acute cases rely on eigenvalue comparisons of triangles, which were first considered by Miyamoto on right triangles. 

Eigenvalue problems on right triangles were also used to establish symmetry (or antisymmetry) of the eigenfunction for the smallest nonzero Neumann eigenvalue of kites (Miyamoto \cite{Mi13}, the author of the present paper \cite{Shot}) and isosceles triangles \cite{LSminN} (in collaboration with Richard Laugesen). It is almost trivial to conclude that the eigenfunction can be assumed symmetric or antisymmetric with respect to a line of symmetry of a domain. It is however very hard to establish which case actually happens. This problem is also strongly connected to the hot-spots conjecture, given that many known results assume enough symmetry to get a symmetric eigenfunction, e.g. Jerison-Nadirashvili \cite{JN00} or Ba\~nuelos-Burdzy \cite{BB99}. As a particular case, the latter paper implies that the smallest nonzero Neumann eigenvalue of a narrow rhombus is antisymmetric with respect to the short diagonal. In order to claim the same for all rhombi one needs to look at the very important hot-spots result due to Atar and Burdzy \cite{AB04}. Their Corollary 1, part ii) can be applied to arbitrary rhombi, but it requires a very sophisticated stochastic analysis argument and a solution of a more complicated hot-spots conjecture to achieve the goal.

As a consequence of the ordering of mixed eigenvalues of right triangles we order first \textit{four} Neumann (and two Dirichlet) eigenvalues of rhombi, depending on their symmetry/antisymmetry. We achieve more than the above mentioned papers, using elementary techniques.

Our result applies to all rhombi not narrower than the ``equilateral rhombus'' composed of two equilateral triangles. This particular case, as well as the square are interesting boundary cases due to the presence of multiple eigenvalues.
\begin{corollary}\label{cor:rhombus}
    For rhombi with the smallest angle $2\alpha>\pi/3$ we have
    \begin{itemize}
      \item $\mu_2$, $\mu_3$, $\mu_4$ and $\lambda_2$ are simple.
      \item $\mu_4<\lambda_1$,
      \item the eigenfunction for $\mu_2$ is antisymmetric with respect to the short diagonal,
      \item the eigenfunction for $\mu_3$ is antisymmetric with respect to the long diagonal,
      \item the eigenfunction for $\mu_4$ is doubly symmetric,
      \item the eigenfunction for $\lambda_2$ is antisymmetric with respect to the short diagonal,
    \end{itemize}
    Furthermore, if $2\alpha<\pi/3$ then the doubly symmetric mode belongs to $\mu_3$, and the mode antisymmetric with respect to the long diagonal can have arbitrarily high index (as $\alpha\to0$).
\end{corollary}

Perhaps the most interesting case of our result about rhombi is that the fourth Neumann eigenvalue of nearly square rhombus is smaller than its smallest Dirichlet eigenvalue (and is doubly symmetric). This strengthens classical eigenvalue comparison results: Payne \cite{Pa55}, Levine-Weinberger \cite{LW86}, Friedlander \cite{Fr91} and Filonov \cite{Fi04} (on smooth enough domains the third Neumann eigenvalue is smaller than the first Dirichlet eigenvalue, while on convex polygons only the second eigenvalue is guaranteed to be below the Dirichlet case, and the third is not larger than it). This type of eigenvalue comparison is traditionally used to derive some conclusions about the nodal set of the Neumann eigenfunction, e.g. an eigenfunction for $\mu_2$ cannot have a nodal line that forms a loop. Recent progress on hot-spots conjecture due to Miyamoto \cite{Mi13} and the author \cite{Shot} relies on such eigenvalue comparisons and similar nodal line considerations. Furthermore, author's forthcoming collaboration \cite{NSY} leverages the improved fourth eigenvalue comparison on rhombi in studying regular polygons.

Our proofs for mixed eigenvalues on triangles are short and elementary, yet a very broad spectrum of techniques is actually needed. Evan though the comparisons look mostly the same, their proofs are strikingly different. Depending on the case, we use: variational techniques with explicitly or implicitly defined test functions, polarization (a type of symmetrization) applied to mixed boundary conditions, nodal domain considerations, or an unknown trial function method (see \cite{LSminN, LSminD}).

\section{Variational approach and auxiliary results}\label{sec:def}
The mixed Dirchlet-Neumann eigenvalues of the Laplacian on a right triangle $T$ with sides of length $L\ge M\ge S$ can be obtained by solving 
\begin{align*}
  \Delta u &= \lambda^D u,\text{ on }T,\\
u &= 0\text{ on }D\subset \{L,M,S\},\\
  \partial_\nu u &= 0\text{ on }\partial T\setminus D.
\end{align*}
The Dirichlet condition imposed on $D$ can be any combination of the triangle's sides, as mentioned in the introduction. For simplicity we denote $\lambda=\lambda^{LMS}$ (purely Dirichlet eigenvalue) and $\mu=\lambda^{\emptyset}$ (purely Neumann eigenvalue).

The same eigenvalues can also be obtained by minimizing the Rayleigh quotient
\begin{align*}
  R[u] = \frac{\int_T |\nabla u|^2}{\int_T u^2}.
\end{align*}
In particular
\begin{align}
  \lambda_1^D &= \inf_{u\in H^1(T), u=0\text{ on }D} R[u],\label{RayleighDir}\\
  \mu_2 &= \inf_{u\in H^1(T),\int_T u=0} R[u],\label{RayleighNeu}
\end{align}
For an overview of the variational approach we refer the reader to Bandle \cite{Ba80} or Blanchard-Br\"uning \cite{BB92}.

For each kind of mixed boundary conditions we have an orthonormal sequence of eigenfunctions and
\begin{align*}
  0<\lambda_1^D<\lambda_2^D\le \lambda_3^D\le \dots \to \infty,
\end{align*}
as long as $D$ is not empty. When $D$ is empty (purely Neumann case) we have
\begin{align*}
  0=\mu_1<\mu_2\le \mu_3\le \mu_4\le\dots\to \infty.
\end{align*}
The sharp inequality $\mu_2<\mu_3$ for all nonequilateral triangles was recently proved by the author \cite{Shot}. Similar result $\lambda_2<\lambda_3$ should hold for purely Dirichlet eigenvalues, but this remains an open problem.

The fact that $\lambda_1^D<\lambda_2^D$ is a consequence of the general smallest eigenvalue simplicity:
\begin{lemma}\label{lem:positive}
  Let $\Omega$ be a domain with Dirichlet condition on $D\ne \emptyset$ and Neumann condition on $\partial \Omega\setminus D$. Then $0<\lambda_1^D<\lambda_2^D$ and the eigenfunction $u_1$ belonging to $\lambda_1^D$ can be taken nonnegative.
\end{lemma}
\begin{proof}
  Suppose $u_1$ is changing sign. Then $|u_1|$ is a different minimizer of the Rayleigh quotient. Any minimizer of the Rayleigh quotient is an eigenfunction (see \cite{Ba80} or a more recent exposition \cite[Chapter 9]{Lnotes}). But $\Delta |u_1|=-\lambda_1 |u_1|\le 0$, hence minimum principle ensures $u_1$ cannot equal zero at any inside point of the domain, giving contradiction. Hence $u_1$ has a fixed sign. If there were two eigenfunctions for $\lambda_1^D$, we could make a linear combination that changes sign, which is not possible. Hence the smallest eigenvalue is simple. Finally, $\lambda_1^D=0$ would imply that $|\partial u_1|= 0$ a.e., hence the eigenfunction is constant. But it equals $0$ on $D$, hence $u_1\equiv 0$.
\end{proof}

If $D_1\subset D_2$ then \eqref{RayleighDir} implies that $\lambda_1^{D_1}\le \lambda_1^{D_2}$. Indeed, any test function $u$ that satisfies $u=0$ on $D_2$, can be used in the minimization of $\lambda_1^{D_1}$. However, the relation between e.g. $\lambda_1^L$ and $\lambda_1^M$ is not clear.

In the second part of the paper we will consider rhombi $R$ created by reflecting a right triangle $T$ four times. 
\begin{lemma}\label{le:extension}
  Let $u$ belong to $\lambda_1^D(T)$ or $\mu_2(T)$. Let $\bar u$ be the extension of $u$ to $R$, that is symmetric with respect to the sides of $T$ with Neumann condition and antisymmetric with respect to the Dirichlet sides. Then $\bar u$ is an eigenfunction of $R$. Furthermore, if $v$ is another eigenfunction of $R$ with the same symmetries as $\bar u$, then $v$ belongs to higher eigenvalue than $\bar u$, or $v= C \bar u$ for some constant $C$.
\end{lemma}
\begin{proof}
  Suppose $v$ is an eigenfunction of $R$ with the same symmetries as $u$. Its restriction to $T$ satisfies Dirichlet and Neumann conditions on the same sides as $u$. It also satisfies the eigenvalue equation pointwise on $T$. Hence $v$ is an eigenfunction on $T$. However, $\lambda_1^D$ and $\mu_2$ are simple, hence $v=C u$, or $v$ belongs to a higher eigenvalue on $T$. 

  The extension $\bar u$ has the same Rayleigh quotient on $R$, as on $T$ (due to symmetries). Hence $\bar u$ can be used as a test function for the lowest eigenvalue on $R$ with the symmetries of $\bar u$. Hence that eigenvalue of $R$ must be smaller or equal to the eigenvalue of $u$ on $T$. However, it cannot be smaller by the argument from the previous paragraph.
\end{proof}

In particular, this lemma implies that 
\begin{align*}
  \lambda_1(R) = \lambda_1^L(T).
\end{align*}
We can also claim that $\mu_2(T)$ equals the smallest Neumann eigenvalue of the rhombus with a doubly symmetric eigenfunction. However, this eigenvalue will not be second on $R$, due to the presence of possibly lower antisymmetric modes (corresponding to $\lambda_1^M(T)$ and $\lambda_1^S(T)$).

\section{Inequalities between mixed eigenvalues of right triangles}\label{sec:mixed}
In this section we prove \autoref{thm:order}. We split the proof into several sections, each treating one or two inequalities. Each section introduces a different technique of proving eigenvalue bounds.

Before we proceed we wish to make a few remarks.

\begin{remark}
All eigenvalues of the right isosceles triangle can be explicitly calculated using eigenfunctions of the square. Obviously $S=M$ in this case, hence some eigenvalue inequalities from \autoref{thm:order} become obvious equalities. Furthermore $\mu_2 = \lambda_1^L$, as can be seen by taking two orthogonal second Neumann eigenfunctions of the unit square with the diagonal nodal lines. One corresponds to $\mu_2$ on the right triangle, the other $\lambda_1^L$.
\end{remark}

\begin{remark}
  Similarly, \textbf{some} of the mixed eigenvalues of the half-of-equilateral triangle can be calculated explicitly using eigenfunctions of the equilateral triangle. In particular $\lambda_1^M = \mu_2$, since the corresponding equilateral triangle has double second Neumann eigenvalue. On the other hand, any mixed case that leads to a mixed case on the equilateral triangle cannot be explicitly calculated. In particular the value of $\lambda_1^{LS}$ on the half-of-equilateral triangle corresponds to equilateral triangle with Dirichlet condition on two sides. The eigenfunction is not trigonometric (as all other known cases), and to the best of our knowledge there is no closed formula for the eigenvalue. 
\end{remark}

For a thorough overview of the explicitly computable cases and the geometric properties of eigenfunction we refer the reader to \cite{GN13} and references therein.

\begin{remark}\label{rem:trap}
  Note that for the trapezium with vertices $(-3,0)$, $(3,0)$, $(3,2)$ and $(0,2)$ imposing the Dirichlet condition on the sloped side leads to smaller eigenvalue than imposing it on the top (numerically).  
\end{remark}

  \subsection{For nonisosceles right triangles: $\lambda_1^S<\lambda_1^M$. Unknown trial function method for isosceles triangles.}\label{sec:unknown}
  In this subsection $O(\beta)$ is an obtuse isosceles triangle with equal sides of length $1$ and aperture angle $2\beta$, with vertices $A, B, C$ equal $(0,h)$, $(\pm\sqrt{1-h^2},0)$, respectively. Let $A(\alpha)$ be an acute isosceles triangle with vertices $A, D, C$ equal $(0,\pm h)$, $(\sqrt{1-h^2},0)$ (aperture angle $2\alpha$). See \autoref{fig:triangles} for both triangles. Finally their intersection is a right triangle $T=AEC$. Note that $h$ and both angles are related by:
  \begin{align*}
  	h=\sin\alpha=\cos \beta,   
  \end{align*}
  and $\beta>\pi/4$.

  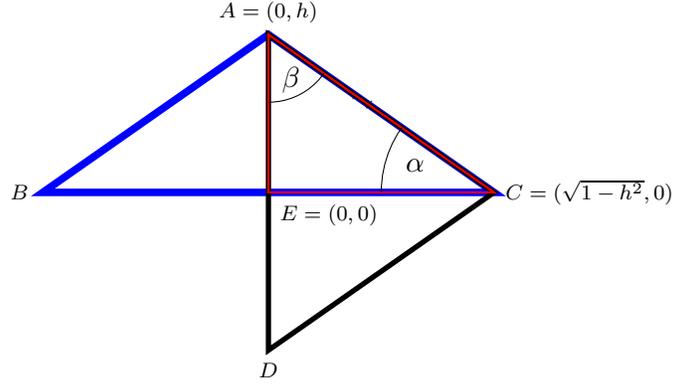
\begin{figure}[t]
  \begin{center}
    \begin{tikzpicture}[scale=3,rotate=-90]
      \draw[blue,line width=1mm] (0,-1) node [left,black] {\tiny $B$}-- (0,1) -- (-0.7,0) node [sloped, above,pos=0.5,black] {\small $1$}-- cycle;
      \draw[black,line width=0.7mm] (-0.7,0) -- (0.7,0) node [below] {\tiny $D$} -- (0,1) -- cycle;
      \draw[red,thick] (-0.7,0) node[above,black] {\tiny $A=(0,h)$} -- (0,0)  node [below right,black] {\tiny $E=(0,0)$} -- (0,1) node[right,black] {\tiny $C=(\sqrt{1-h^2},0)$} -- cycle;
      \clip (-0.7,0) -- (0,0) -- (0,1) -- cycle;
      \draw (-0.7,0) circle (0.3);
      \draw (-0.5,0.1) node {\small $\beta$};
      \draw (0,1) circle (0.5);
      \draw (-0.12,0.65) node {\small $\alpha$};
    \end{tikzpicture}
  \end{center}
  \caption{Obtuse isosceles triangle $O(\beta)=ABC$ and acute isosceles triangle $A(\alpha)=ADC$.}
  \label{fig:triangles}
\end{figure}
We need the following three lemma:

  \begin{lemma}\label{lemobt}
  For $\beta>\pi/4$
  \begin{align*}
    \mu_2(O(\beta))<\frac{\pi^2}{4h^2(1-h^2)}
  \end{align*}
  And the bound saturates for the right isosceles triangle ($h^2=1/2$ or $\beta=\pi/4$).
\end{lemma}
\begin{proof}
  Note that $h^2< 1/2$. Take the second eigenfunction for the right isosceles triangle $(0,1)$, $(\pm 1,0)$ and deform it linearly to fit $O(\beta)$. That is take
  \begin{align*}
    \varphi=\sin(\pi x/2)\cos(\pi y/2)
  \end{align*}
  and compose with the linear transformation $L(x,y)=(x/\sqrt{1-h^2},y/h)$. Resulting function can be used as a test function for $\mu_2(O(\beta))$
  \begin{align*}
    \mu_2(O(\beta))\le \frac{\int_{T(\beta)} |\nabla (\varphi\circ L)|^2}{\int_{T(\beta)} |\varphi\circ L|^2}=\frac{\pi^2+16h^2-8}{4h^2(1-h^2)}< \frac{\pi^2}{4h^2(1-h^2)}
  \end{align*}
\end{proof}

\begin{lemma}\label{interval}
  Let $u$ be any antisymmetric function on $A(\alpha)$ (so that $u(x,-y)=-u(x,y)$). Then
  \begin{align*}
    \int_{A(\alpha)}u_y^2> \frac{\pi^2}{4h^2} \int_{A(\alpha)} u^2.
  \end{align*}
\end{lemma}
\begin{proof}
  Note that for fixed $x$ function $u(x,\cdot)$ is odd, hence it can be used as a test function for the second Neumann eigenvalue on any vertical interval contained in the triangle $A(\alpha)$. We get the largest interval $[-h,h]$ when $x=0$. Hence 
  \begin{align*}
    \int_{[-c_x,c_x]} u_y^2(x,y)\,dy\ge \mu_2([-c_x,c_x])\int_{[-c_x,c_x]} u^2(x,y)\,dy \ge \frac{\pi^2}{4h^2} \int_{[-c_x,c_x]} u^2(x,y)\,dy.
  \end{align*}
  Integrate over $x$ to get the result.
\end{proof}

A special case of \cite[Corollary 5.5]{LSminN}, noting that $\alpha<\beta$, can be stated as
\begin{lemma}\label{unknown}
     Let $u$ be the eigenfunction belonging to $\mu_2(A(\alpha))$. Then $\mu_2(O(\beta))< \mu_2(A(\alpha))$ if
    \begin{align}
      \frac{\int_{A(\alpha)}u_y^2}{\int_{A(\alpha)}u_x^2}>\tan^2(\beta).\label{eq:cond}
    \end{align}
\end{lemma}
Suppose the condition \eqref{eq:cond} is false (hence we cannot conclude that $\mu_2(O(\beta))< \mu_2(A(\alpha))$). That is
 \begin{align*}
   \int_{A(\alpha)}u_y^2\le \tan^2(\beta)\int_{A(\alpha)}u_x^2
    \end{align*}
Then
\begin{align*}
  \mu_2(A(\alpha))&=\frac{\int_{A(\alpha)} u_x^2+u_y^2}{\int_{A(\alpha)} u^2}\ge \left( 1+\frac1{\tan^2(\beta)} \right)\frac{\int_{A(\alpha)} u_y^2}{\int_{A(\alpha)} u^2}
  > \frac{1}{\sin^2(\beta)} \frac{\pi^2}{4h^2}=
  \\&=\frac{\pi^2}{4h^2(1-h^2)}>\mu_2(O(\beta)),
\end{align*}
where the last inequality in the first line follows from \autoref{interval}, while the inequality in the second line follows from \autoref{lemobt}.

Therefore, regardless if we can apply \autoref{unknown} or not (condition \eqref{eq:cond} is true or false), we get 
\begin{align*}
\mu_2(O(\beta))<\mu_2(A(\alpha)).
\end{align*}

Since $O(\beta)$ is obtuse and isosceles, \cite[Theorem 3.2]{LSminN} implies that $\lambda_1^S(T) = \mu_2(O(\beta))$.
The eigenfunction for $\lambda_1^M(T)$ extends to an antisymmetric eigenfunction on $A(\alpha)$, hence $\mu_2(A(\alpha))\le \lambda_1^M(T)$. 

Therefore we proved that $\lambda_1^S< \lambda_1^M$ for any nonisosceles right triangle.

  \subsection{For right triangles: $\lambda_1^M<\mu_2$ if and only if $\alpha>\pi/6$. Comparison of Neumann eigenfunctions of an isosceles triangle. }

  In this section we will use the notation introduced in \cite[Section 3]{LSminN}. All isosceles triangles can be split into equilateral triangles, subequilateral triangles (with angle between equal sides less than $\pi/3$), and superequilateral (with the angle above $\pi/3$).

  Note that $\alpha=\pi/6$ means that we are working with a half of an equilateral triangle. The eigenvalues are explicit and $\lambda_1^M=\mu_2$.

  Mirroring a right triangle along middle side $M$ gives a superequilateral triangle if and only if $\alpha>\pi/6$. Any superequilateral triangle has antisymmetric second Neumann  eigenfunction \cite[Theorem 3.2]{LSminN} and simple second eigenvalue \cite{Mi13} equal $\lambda_1^M(T)$. This proves that $\lambda_1^M(T)<\mu_2(T)$. 
  
  At the same time any subequilateral triangle has symmetric second eigenfunction \cite[Theorem 3.1]{LSminN}, with simple eigenvalue \cite{Mi13} equal $\mu_2(T)$. Hence $\lambda_1^M>\mu_2(T)$ if $\alpha>\pi/6$.

  \subsection{For right triangles: $\mu_2<\lambda_1^L<\lambda_1^{MS}$. Variational approach and domain monotonicity.}
  Assume that the right triangle $T$ has vertices $(0,0)$, $(1,0)$ and $(0,b)$. We can use four such right triangles to build a rhombus $R$. Then $\lambda_1^L=\lambda_1(R)$, since reflected eigenfunction for $\lambda_1^L$ is nonnegative and satisfies Dirichlet boundary condition on $R$ (see \autoref{le:extension}). Hooker and Protter \cite{HP61} proved the following lower bound for the ground state of rhombi
  \begin{align}\label{rhombus}
  \lambda_1^L=\lambda_1(R)\ge \frac{\pi^2(1+b)^2}{4b^2}.
\end{align}
We need to prove an upper bound for $\mu_2$ that is smaller than this lower bound.

Consider two eigenfunctions of the right isosceles triangle with vertices $(0,0)$, $(1,0)$ and $(0,1)$
\begin{align*}
  \varphi_1(x,y)=\cos(\pi y)-\cos(\pi x),\\
  \varphi_2(x,y)=\cos(\pi y)\cos(\pi x).
\end{align*}
The first one belongs to $\mu_2$ and is antisymmetric, the second belongs to $\mu_3$ and is symmetric. In fact all we need is that these integrate to $0$ over the right isosceles triangle. Consider a linear combination of the linearly deformed eigenfunctions
\begin{align*}
  f(x,y)=\varphi_1(x,y/b)-(1-b)\varphi_2(x,y/b),
\end{align*}
where $0<b<1$. This function integrates to $0$ over the right triangle $T$, hence it can be used as a test function for $\mu_2$ in \eqref{RayleighNeu}. As a result we get the following upper bound
\begin{align}\label{isobound}
  \mu_2\le \frac{3\pi^2 ((b-1)^2+2)(b^2+1)-64(b-1)^2(b+1) }{3b^2((b-1)^2+4)}
\end{align}
Note that when $b=1$ bounds \eqref{rhombus} and \eqref{isobound} reduce to the same value. In fact $\lambda_1^L=\mu_2$ in this case (right isosceles triangle).

Moreover
\begin{align*}
  &\frac{3\pi^2 ((b-1)^2+2)(b^2+1)-64(b-1)^2(b+1) }{3b^2((b-1)^2+4)}
  -
  \frac{\pi^2(1+b)^2}{4b^2}=
  \\&\qquad\qquad=
  \frac{(b-1)^2}{12b^2( (b-1)^2+4)} (9\pi^2 b^2-(256+6\pi^2)b+21\pi^2-256),
\end{align*}
and the quadratic expression in $b$ is negative for $b\in(0,1)$. Therefore $\mu_2<\lambda_1^L$.

 For right triangles, $\lambda_1^L$ is the same as $\lambda_1$ for a rhombus built from four triangles, while $\lambda_1^{MS}$ is the same as $\lambda_1$ of a kite built from two right triangles. Sharpest angle of the kite is the same as the acute angle of the rhombus. We can put the kite inside of the rhombus by putting the vertex of the sharpest angle at the vertex of the rhombus. Therefore $\lambda_1^L<\lambda_1^{MS}$ by domain monotonicity (take the eigenfunction of the kite, extend with $0$, and use as trial function on the rhombus).

\subsection{For arbitrary triangle: $\min\{\lambda_1^S,\lambda_1^M,\lambda_1^L\}<\mu_2\le \lambda_{1}^{MS}$. Nodal line consideration and eigenvalue comparisons.}

To get a lower bound for $\mu_2$ we will define a trial function based on the Neumann eigenfunction, without knowing its exact form, and use it as a trial function for a mixed eigenvalue problem. Note that the eigenfunction of $\mu_2$ has exactly two nodal domains, by Courant's nodal domain theorem (\cite[Sec. V.5, VI.6]{CH53}) and orthogonality to the first constant eigenfunction. Hence the closure of at least one of these nodal domains must have empty intersection with the interior of one of the sides (nodal line might end in a vertex, but the eigenfunction must have fixed sign on at least one side). Let us call this side $D$ and consider $\lambda_1^D$. Let $u$ be the eigenfunction of $\mu_2$ restricted to the nodal domain not intersecting $D$. Extend $u$ with $0$ to the whole triangle $T$. We get a valid trial function for $\lambda_1^D$. Hence $\min\{\lambda_1^S,\lambda_1^M,\lambda_1^L\}\le \lambda_1^D<\mu_2$.

Note that we already proved that for right triangles $\lambda_1^S<\lambda_1^M$, $\mu_2<\lambda_1^L$, and $\lambda_1^M < \mu_2$ if and only if smallest angle $\alpha>\pi/6$. Hence for right triangles the minimum can be replaced by $\lambda_1^S$, or even $\lambda_1^M$ if $\alpha>\pi/6$.

Note also, that the result of this section generalizes to arbitrary polygons.
\begin{lemma}
  The smallest nonzero Neumann eigenvalue on a polygon with $2n+1$ or $2n+2$ sides is bounded below by the minimum of all mixed Dirichlet-Neumann eigenvalues with Dirichlet condition applied to at least $n$ consecutive sides.

Furthermore, for arbitrary domain, the Neumann eigenvalue is bounded below by the infimum over all mixed eigenvalue problems with half of the boundary length having Dirichlet condition applied to it.
\end{lemma}

As in the previous section, $\lambda_1^{MS}$ equals $\lambda_1$ of a kite built from two triangles. The Neumann eigenfunction for $\mu_2$ extended to the kite gives a symmetric eigenfunction of the kite. Given that $\mu_2$ and $\mu_3$ of the kite together can have at most one antisymmetric mode (see \cite[Lemma 2.1]{Shot}), we conclude that $\mu_2$ of the triangle is no larger than $\mu_3$ of the kite. Levine and Weinberger \cite{LW86} proved that $\lambda_1\ge \mu_3$ for any convex polygon, including kites, giving us the required inequality.

\subsection{Proof of: $\lambda_1^{LS}<\lambda_1^{LM}$. Symmetrization of isosceles triangles.}
\begin{figure}[t]
  \begin{center}

      \hspace{\fill}
      \begin{tikzpicture}[baseline=0,scale=1.3]
	\draw[thick,red] (-2,0) coordinate (a) -- (0,0.8) coordinate (c) -- ($-1*(c)$) coordinate (b) -- cycle;
	\draw (-2,0) -- (2,0) -- (c) -- cycle;
	\draw[thick,green!50!black,->] (2,0) -- (b);
	\coordinate (d) at ($0.5*(a)+0.5*(c)$);
	\coordinate (e) at ($(d)!(b)!90:(a)$);
	\draw[dotted] ($(d)+0.2*(d)-0.2*(e)$) -- ($(e)-0.2*(d)+0.2*(e)$);
      \end{tikzpicture}
      \hspace{\fill}
      \begin{tikzpicture}[baseline=0,scale=1.3]
	\draw[thick,red] (-2,0) coordinate (a) -- (0,1.5) coordinate (c) -- ($-1*(c)$) coordinate (b) -- cycle;
	\coordinate (d) at ($0.5*(a)+0.5*(c)$);
	\coordinate (e) at ($(d)!(b)!90:(a)$);
	\draw[dotted] ($(d)+0.2*(d)-0.2*(e)$) -- ($(e)-0.2*(d)+0.2*(e)$);
	\draw[thick,blue,dashed] (c) -- ($(e)+(e)-(b)$) coordinate (f) -- (a) -- cycle;
	\draw (-2,0) -- (2,0) -- (c) -- cycle;
	\draw[thick,green!50!black,->] (2,0) -- (f);
      \end{tikzpicture}
      \hspace{\fill}

  \end{center}
  \caption{Acute isosceles triangle (thick red line) and obtuse isosceles triangle (thin black line) generated by the same right triangle (their intersection). Two cases of continuous Steiner symmetrization based on the shape of the acute isosceles triangle: subequilateral on the left, superequilateral on the right.}
  \label{fig:domains2}
\end{figure}
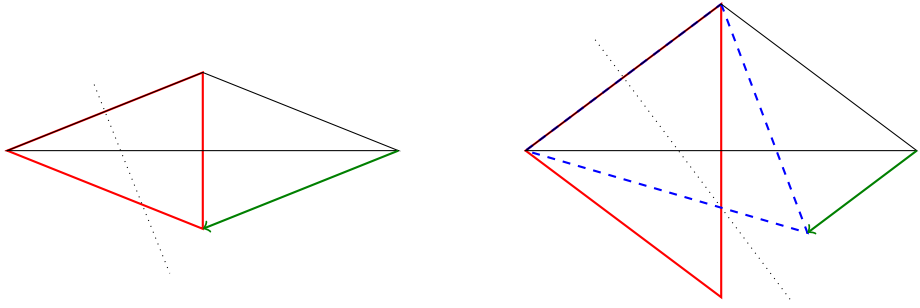
To prove this inequality we will use a symmetrization technique called the continuous Steiner symmetrization introduced by P\'olya and Szeg\"o \cite[Note B]{PS51}, and studied by Solynin \cite{So90,So12} and Brock \cite{Br95}. The author already used this technique for bounding Dirichlet eigenvalues of triangles in \cite{Siso}. See Section 3.2 in the last reference for detailed explanation. The most important feature of the transformation is that if one can map one domain to another using that transformation, then the latter has smaller Dirichlet eigenvalue.

Note that mirroring a right triangle along the middle side shows that $\lambda_1^{LS}$ of the triangle equals $\lambda_1$ of an acute isosceles triangle. Similarly, $\lambda_1^{LM}$ equals $\lambda_1$ for an obtuse isosceles triangle. We need to show that the acute isosceles triangle has smaller Dirichlet eigenvalue. \autoref{fig:domains2} shows both isosceles triangles. Position the isosceles triangles as on the figure, and perform the continuous Steiner symmetrization with respect to the line perpendicular to the common side.

If the acute isosceles triangle is subequilateral (vertical side is the shortest, left picture on \autoref{fig:domains2}), then before we fully symmetrize the obtuse triangle, we will find the acute one. The arrow on the figure shows how far we should continuously symmetrize. Therefore the acute isosceles triangle has smaller eigenvalue. 

If the acute isosceles triangle is superequilateral (vertical side is the longest, right picture on \autoref{fig:domains2}), then we first reflect the acute triangle across the symmetrization line, then perform continuous Steiner symmetrization. Again, we get that the acute isosceles triangle has smaller eigenvalue.

Note that this case seems similar to \autoref{sec:unknown}. However, we get to use symmetrization technique due to Dirichlet boundary, while in the other section we had to us a less powerful, but more broadly applicable, unknown trial function method. 

\subsection{For arbitrary triangles: $\lambda_1^{MS}<\lambda_1^{LS}<\lambda_1^{LM}$. Polarization with mixed boundary conditions.}

For this inequality we use another symmetrization technique called polarization. It was used by Dubinin \cite{Du93}, Brock and Solynin \cite{Br95b,So96,BS00,So12}, Draghici \cite{Dr03}, and the author \cite{Siso} to study various aspects of spectral and potential theory of the Laplacian. As with other kinds of symmetrization, if one can map a domain to some other domain, then the latter has smaller eigenvalue.  

  Polarization involves a construction of a test function for the lowest Dirichlet eigenvalue of the transformed domain from the nonnegative eigenfunction of the original domain. We choose to deemphasize the geometric transformation involved, and focus on the transplanted eigenfunction. In fact we transform a triangle into itself, so that boundary conditions change the way we need for the proof. 
  
  Note that unlike in other applications of polarization mentioned above, we apply it to mixed boundary conditions. We showed in \autoref{lem:positive} that the lowest mixed eigenvalue is simple and has a nonnegative eigenfunction. We will use this eigenfunction to create an eigenfunction on transformed domain.

  \begin{figure}[t]
    \begin{center}

      \hspace{\fill}
      \begin{tikzpicture}[baseline=0]
	\draw[thick,red] (0,0) node [left,black] {\tiny $0$} -- (45:5) coordinate (a) node [above,black] {\tiny $A$}-- (0:4) coordinate (b) node [below,black] {\tiny $B$};
	 \draw[very thick, densely dotted, black] (0,0) -- (b);
	 \draw[blue,dashed] (0,0) -- (0:5) coordinate (c) -- (45:4) coordinate (d);
	 \draw[dotted] (0,0) -- ($(a)!0.5!(c)$) coordinate (e);
	 \draw ($0.5*(d)+0.3*(a)+0.2*(e)$) node {\scriptsize $w$};
	 \draw ($0.33*(d)+0.33*(e)$) node {\scriptsize $v$};
	 \draw ($0.33*(b)+0.33*(e)$) node {\scriptsize $u$};
      \end{tikzpicture}
      \hspace{\fill}
      \begin{tikzpicture}[baseline=0]
         \draw[dashed,red] (0,0) -- (45:5) coordinate (a) -- (0:4) coordinate (b) -- (0,0);
	 \draw[blue,thick]  (b) -- (0:5) coordinate (c) -- (45:4) coordinate (d) -- (0,0);
	 \draw[very thick, densely dotted,black] (0,0) -- (b);
	 \draw[dotted] (0,0) -- ($(a)!0.5!(c)$) coordinate (e);
	 \draw ($0.5*(b)+0.3*(c)+0.2*(e)$) node {\scriptsize $w$};
	 \draw ($0.33*(d)+0.33*(e)$) node {\scriptsize $\min(\bar u,\bar v)$};
	 \draw ($0.33*(b)+0.33*(e)$) node {\scriptsize $\max(\bar u,\bar v)$};
      \end{tikzpicture}
      \hspace{\fill}

    \end{center}
    \caption{A triangle with Dirichlet condition on two sides ($OA$ and $AB$), and the same triangle reflected along the bisector of one the angle $AOB$ (with Dirichlet condition on blue lines). The eigenfunction from the left picture can be rearranged into a test function on the right picture, preserving boundary conditions, as long as $|OB|<|OA|$. }
    \label{fig:domains}
  \end{figure}
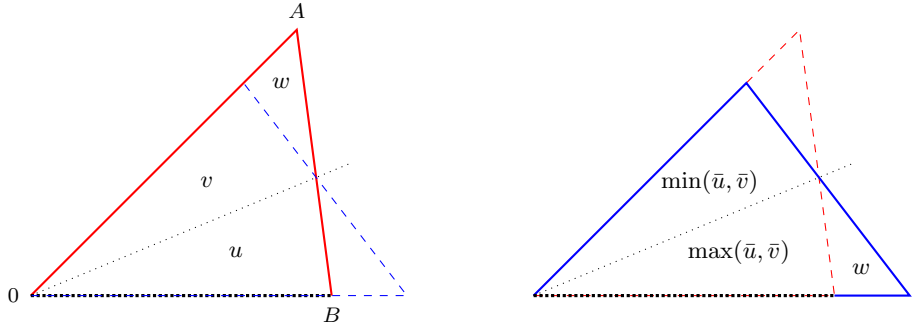

  Let $T$ be a triangle. We apply the Dirichlet condition on two sides. Without loss of generality let us assume we do this on $L$ and $S$ (see left picture on \autoref{fig:domains}). Let the eigenfunction for $\lambda_1^{LS}$ equal $u$, $v$ and $w$ on the parts of the domain shown on the figure. Let $\bar u$ and $\bar v$ be the symmetric extensions of $u$ and $v$ along the bisector of their common angle (dotted line). We rearrange the parts to fit the dashed triangle, as on the right picture of the same figure.

We need to check that the rearranged trial function is continuous on the blue triangle, and it satisfied the Dirichlet conditions on correct sides. It is crucial in this step that $u$ snd $v$ are nonnegative.

On the dotted line $\bar u=u=v=\bar v$, due to continuity of the original eigenfunction. On the dashed line $\max(\bar u,\bar v)=\bar v$, since $u$ satisfies the Dirichlet condition there. Hence the test function is continuous on the dashed line due to continuity of the original eigenfunction on the interface of $v$ and $w$.

On the long sloped blue side of the triangle $\min(\bar u,\bar v)=v=0$. On the part of the short sloped side to the right of the dashed line we have $w=0$. Finally, the part to the right satisfies $\min(\bar u,\bar v)=u=0$. Therefore the trial function satisfied the Dirichlet boundary condition on the middle and short sides of the blue right triangle. 

We polarized the red right triangle into the blue right triangle (same shape), but the Dirichlet conditions moved from $LS$ to $MS$, as we needed. In fact there is an additional part of the third side with Dirichlet condition applied, ensuring strict inequality in our result.

The only assumptions we needed in the construction is that the $|OB|>|OA|$ and Dirichlet condition on $AB$. The same conditions can be enforced in the comparison of $\lambda_1^{LM}$ and $\lambda_1^{LS}$.

\section{Proof of \autoref{cor:rhombus}}

 Four copies of the same right triangle can be used to build a rhombus. Let $R$ denote the rhombus, and $T$ the right triangle that can be used to build $R$ (see \autoref{fig:rhombus}).
    The order of eigenfunctions that is claimed in the theorem follows from the order of the eigenvalues for the triangle, \autoref{thm:order}. We need to show that there are no other eigenfunctions intertwined with the ones listed. All eigenfunctions can be taken symmetric or antisymmetric with respect to each diagonal. 
    
    By \cite[Lemma 2.1]{Shot}, the eigenspace $S$ of $\mu_2(R)$ and $\mu_3(R)$ can contain at most one eigenfuncion antisymmetric with respect to a given diagonal. Therefore if there are more than 2 eigenfunctions in $S$, the extra ones must be doubly symmetric. But the lowest doubly symmetric mode equals $\mu_2(T)$ and it is larger than $\lambda_1^M$ and $\lambda_1^S$ (these two eigenvalues generate antisymmetric eigenfunctions on $R$). Therefore $\mu_2(R)$ and $\mu_3(R)$ are simple. 
    
\begin{figure}[t]
  \begin{center}
\begin{tikzpicture}[scale=1.5,baseline=0]
  \draw (0,0) -- (1,0.7) -- (2,0) -- (1,-0.7) node [below] {\tiny $\mu_2(R)=\lambda_1^S(T)$}-- cycle;
  \draw[red] (1,0.7) -- (1,-0.7);
  \draw[blue,dashed] (0,0) -- (2,0);
\end{tikzpicture}
\begin{tikzpicture}[scale=1.5,baseline=0]
  \draw (0,0) -- (1,0.7) -- (2,0) -- (1,-0.7) node [below] {\tiny $\mu_3(R)=\lambda_1^M(T)$}-- cycle;
  \draw[red] (0,0) -- (2,0);
  \draw[blue,dashed] (1,0.7) -- (1,-0.7);
\end{tikzpicture}
\begin{tikzpicture}[scale=1.5,baseline=0]
  \draw (0,0) -- (1,0.7) -- (2,0) -- (1,-0.7) node [below] {\tiny $\mu_4(R)=\mu_2(T)$}-- cycle;
  \clip (0,0) -- (1,0.7) -- (2,0) -- (1,-0.7);
  \draw[blue,dashed] (0,0) -- (2,0);
  \draw[blue,dashed] (1,0.7) -- (1,-0.7);
  \draw[red] (0,0) circle (0.65);
  \draw[red] (2,0) circle (0.65);
\end{tikzpicture}
\begin{tikzpicture}[scale=1.5,baseline=0]
  \draw[red] (0,0) -- (1,0.7) -- (2,0) -- (1,-0.7) node [below,black] {\tiny $\lambda_2(R)=\lambda_1^{LS}(T)$}-- cycle;
  \draw[red] (1,0.7) -- (1,-0.7);
  \draw[blue,dashed] (0,0) -- (2,0);
\end{tikzpicture}
  \end{center}
  \caption{Neumann eigenfunctions for nearly square rhombi (red/solid lines - antisymmetry/nodal line, blue/dashed lines - symmetry). Note that $\mu_2$, $\mu_3$ and $\lambda_2$ correspond to mixed eigenvalues on a right triangle, while $\mu_4$ corresponds to the Neumann mode on the same triangle (the position of the nodal arcs for $\mu_4$ is based on numerical computations).}
  \label{fig:rhombus}
\end{figure}
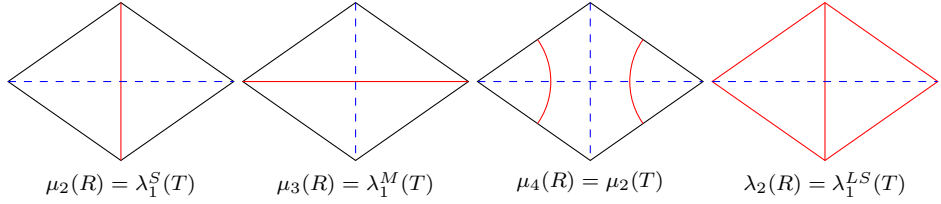
    Suppose an antisymmetric mode belongs to $\mu_4(R)$. Then it belongs to one of the following eigenvalues on $T$: $\lambda_1^{MS}$, $\lambda_k^S$ or $\lambda_k^M$ with $k\ge 2$. But these are larger than $\mu_2(T)$. Hence $\mu_4(R)$ consists of only doubly symmetric modes. On the other hand $\mu_2(T)$ is simple, hence $\mu_4(R)$ is also simple. Finally $\lambda_1(R)=\lambda_1^L>\mu_2(T)=\mu_4(R)$.
  \begin{remark}
    If $2\alpha=\pi/3$ then the rhombus has antisymmetric $\mu_2$. But then there is a double eigenvalue $\mu_3$ which equals to $\mu_2$ for equilateral triangle. Hence the above theorem fails for $\alpha\le \pi/6$. When $\alpha<\pi/6$, argument involving subequilateral triangles shows that $\mu_3$ is doubly symmetric. When $\alpha$ is very small the mode antisymmetric with respect to the long diagonal may have arbitrarily high index. 
  \end{remark}
  \begin{remark}
    Numerical results suggest that the eigenfunction for $\mu_5$ is either doubly antisymmetric for nearly square rhombi (same as $\lambda_1^{MS}$), or antisymmetric along the short diagonal with one more nodal line in each half (same as $\lambda_2^S$). The second doubly symmetric mode is always larger than the latter, but can be smaller than the former.
   The eigenfunction for $\lambda_3$ is either antisymmetric with respect to the long diagonal, or doubly symmetric. The numerical experiments suggest that the first case holds.
  \end{remark}
  \begin{remark}
    P\"utter \cite{P91} showed that the nodal line for the second Neumann eigenfunction for certain doubly symmetric domains is on the shorter axis of symmetry. However, rhombi do not satisfy the conditions required for these domains. 
  \end{remark}

\end{document}